\title{Convergent sequences\\ of combinatorial submodular setfunctions}
\author{
Kristóf Bérczi\thanks{MTA-ELTE Matroid Optimization Research Group and HUN-REN–ELTE Egerváry Research Group, Department of Operations Research, Eötvös Loránd University, and HUN-REN Alfréd Rényi Institute of Mathematics, Budapest, Hungary. Email: \texttt{kristof.berczi@ttk.elte.hu}.}
\and
Márton Borbényi\thanks{Department of Computer Science, Eötvös Loránd University  and HUN-REN Alfréd Rényi Institute of Mathematics, Budapest, Hungary. Email: \texttt{marton.borbenyi@ttk.elte.hu}.}
\and
László Lovász\thanks{HUN-REN Alfréd Rényi Institute of Mathematics, Budapest, Hungary. Email: \texttt{laszlo.lovasz@ttk.elte.hu}.}
\and
László Márton Tóth\thanks{HUN-REN Alfréd Rényi Institute of Mathematics, Budapest, Hungary. Email: \texttt{toth.laszlo.marton@renyi.hu}.}
}
\date{}

\documentclass[10pt]{article}
\usepackage{fullpage}
\usepackage[utf8]{inputenc}
\usepackage{amsthm,amssymb,amsmath,bbm}
\usepackage{soul}
\usepackage{caption,subcaption}
\usepackage{epsfig,graphicx,graphics,color,tikz,tikz-cd,tcolorbox}
\usepackage{makecell}
\usepackage{enumerate,enumitem}
\usepackage[sort]{cite}
\usepackage{hyperref}
\usepackage{todonotes}
\usepackage{mathtools}
\usepackage[normalem]{ulem}
\hypersetup{
    colorlinks=true,       
    linkcolor=blue,        
    citecolor=magenta,         
    filecolor=magenta,     
    urlcolor=cyan,         
    linktoc=all
}

\makeatletter
\newcommand{\leqnomode}{\tagsleft@true\let\veqno\@@leqno}
\makeatother

\makeatletter
 \newcommand{\linkdest}[1]{\Hy@raisedlink{\hypertarget{#1}{}}}
\makeatother

\def\final{0}  
\def\iflong{\iffalse}
\ifnum\final=0  
\newcommand{\laci}[1]{{\color{violet}[{ \textbf{LLaci:}  #1}]\marginpar{\color{violet}*}}}
\newcommand{\tlaci}[1]{{\color{purple}[{ \textbf{TLaci:}  #1}]\marginpar{\color{purple}*}}}
\newcommand{\marci}[1]{{\color{blue}[{ \textbf{Marci:}  #1}]\marginpar{\color{blue}*}}}
\newcommand{\kristof}[1]{{\color{red}[{ \textbf{Kristóf:}  #1}]\marginpar{\color{red}*}}}

\else 
\newcommand{\laci}[1]{}
\newcommand{\tlaci}[1]{}
\newcommand{\marci}[1]{}
\newcommand{\kristof}[1]{}
\fi

\usepackage{stackengine}
\stackMath

\theoremstyle{plain}
\newtheorem{thm}{Theorem}[section]
\newtheorem{lem}[thm]{Lemma}
\newtheorem{cor}[thm]{Corollary}
\newtheorem{cl}[thm]{Claim}
\newtheorem{prop}[thm]{Proposition}

\newtheorem{prob}[thm]{Problem}

\theoremstyle{definition}

\newtheorem{ex}[thm]{Example}

\newtheorem{rem}[thm]{Remark}

\newenvironment{proof*}[1]{\medskip\noindent{\bf Proof of #1.}}{\hfill$\square$\medskip}

\newcommand{\R}{\mathbb{R}}

\newcommand{\PER}{\angle}

\newcommand*\diff{\mathop{}\!\mathrm{d}}

\newcommand{\bP}{\mathbf{P}}
\newcommand{\bA}{\mathbf{A}}
\newcommand{\bB}{\mathbf{B}}
\newcommand{\bT}{\mathbf{T}}
\newcommand{\bV}{\mathbf{V}}
\newcommand{\bU}{\mathbf{U}}
\newcommand{\bQ}{\mathbf{Q}}

\newcommand{\bbE}{\mathbb{E}}
\newcommand{\bbR}{\mathbb{R}}
\newcommand{\bbZ}{\mathbb{Z}}
\newcommand{\bbF}{\mathbb{F}}
\newcommand{\bbN}{\mathbb{N}}

\newcommand{\cB}{\mathcal{B}}

\newcommand{\cF}{\mathcal{F}}

\newcommand{\FF}{\mathcal{F}}

\newcommand{\LL}{\mathcal{L}}
\newcommand{\MM}{\mathcal{M}}
\newcommand{\PP}{\mathcal{P}}

\newcommand{\cP}{\mathcal{P}}

\newcommand{\cX}{\mathcal{X}}

\newcommand{\tc}{\tilde{c}}

\newcommand{\Af}{\mathfrak{A}}

\newcommand{\fg}{\varphi}
\newcommand{\eps}{\varepsilon}
\newcommand{\Haus}{{\text{\rm Haus}}}
\newcommand{\hto}{\stackrel{\Haus}{\longrightarrow}}

\long\def\ignore#1{}
\usepackage{titling}

\newcommand*{\claimproofname}{Proof of claim.}
\newenvironment{claimproof}[1][\claimproofname]{\begin{proof}[#1]}{\end{proof}}

\newlength{\bibitemsep}
\newlength{\bibparskip}
\let\oldthebibliography\thebibliography
\renewcommand\thebibliography[1]{%
  \oldthebibliography{#1}%
  \setlength{\parskip}{\bibitemsep}%
  \setlength{\itemsep}{\bibparskip}%
}

\begin{document}
\maketitle
\thispagestyle{empty}

\begin{abstract} 
To illustrate that the notion of convergence of submodular function sequences fits reasonably into the limit theory of graphs, we describe several classes of matroids and other submodular setfunctions for which convergence of appropriate sequences can be proved. Some of the proofs are surprisingly nontrivial.
\medskip

\end{abstract}

\tableofcontents
\newpage
\pagenumbering{arabic}
\setcounter{page}{1}
 \newpage

\section{Introduction}
\label{sec:intro}

A limit theory for matroids, and more generally for submodular setfunctions, has been developed in~\cite{BBLTconv}. It was shown in~\cite{BBLTgraphing} that the theory is in line with the notion of local-global convergnece of bounded-degree graphs. In this paper we show that the limit theory applies to a variety of further combinatorially interesting submodular setfunctions, like the rank functions of linear spaces, cut capacity functions, homomorphism density functions, and matroids of many dense graphs. These examples are mentioned in~\cite{BBLTconv} without complete proofs, and the goal of these notes is to elaborate those proofs.

\section{Preliminaries}
\label{sec:preliminaries}
\subsection{Definitions}
\label{sec:definitions}

Let $\cF$ be a family of subsets of the set $J$, which is not necessarily finite. A {\it setfunction} on $(J,\cF)$ is a function $\varphi\colon\cF\to\R$. A {\it lattice family} is a family $\cF$ of subsets of a set $J$ closed under union and intersection. A setfunction $\varphi$ on a lattice family $(J,\cF)$ is {\it submodular}, if 
\begin{equation*}
  \varphi(X)+\varphi(Y) \geq \varphi(X \cap Y)+\varphi(X \cup Y)  
\end{equation*}
for all $X,Y\in\FF$. Throughout paper, we assume that $\fg(\emptyset)=0$. In combinatorics, one of the most important examples of submodular setfunctions is the rank function of a matroid.

For a setfunction $\fg$ on a set algebra $(J,\cF)$ and $k\in\bbZ_+$, we define a {\it $k$-quotient} of $\fg$ by $\psi=\fg\circ F^{-1}$, where $F\colon J\to[k]$ is a measurable map. The {\it $k$-quotient set} $Q_k(\fg)$ is the set of all $k$-quotients of $\fg$. The set $Q_k(\fg)$ is a subset of the finite dimensional real linear space $\R^{2^k}$, with the two coordinates $\fg\circ F^{-1}(\emptyset)=0$ and $\fg\circ F^{-1}(J)=\fg(J)$ being fixed. For any sequence $\bA=(A_1,\dots,A_k)\in\cF^k$, we define $\fg/\bA\colon 2^{[k]}\to\R$ by $(\fg/\bA)(I)=\fg\big(\cup_{i\in I} A_i\big)$ for $I\subseteq[k]$. Then quotients could be defined as setfunctions $\fg/\bA$,
where $\bA=(A_1,\dots,A_k)$ is a partition of $J$ into sets in $\cF$.

We say that the sequence $(\fg_n\colon n\in\bbZ_+)$ of finite setfunctions is {\it quotient-convergent}, if for every $k\in\bbZ_+$, the quotient sets $(Q_k(\fg_n)\colon n\in\bbZ_+)$ form a Cauchy sequence in the Hausdorff distance $d^\Haus$ on the $2^k$-dimensional Euclidean space. It does not matter, of course, which norm on $\R^{2^k}$ is used to define the Hausdorff distance, but for conformity with the infinite setting, we use the $\ell^\infty$ norm, simply denoted by $\|\cdot\|$. We say that $(\fg_n\colon  n\in\bbZ_+)$ {\it quotient-converges to} $\fg$ for some setfunction $\fg$ if the sequence $(Q_k(\fg_n)\colon n\in\bbZ_+)$ converges to $Q_k(\fg)$ in the Hausdorff distance. It is proved in~\cite{BBLTconv} that every convergent sequence of finite submodular setfunctions has a limit object in the form of a submdolar setfunction on a standard Borel space. 

\subsection{Graph limits}
\label{sec:graph}

We need basic notions and results from the theory of convergence of dense graphs, based on~\cite{BCLSV1},\cite{BCLSV2} and~\cite{LSzlimit}. The relationship between convergence of bounded-degree graph sequences and their associated matroids is discussed in~\cite{BBLTgraphing}. See also~\cite{lovasz2012large} for further developments in both settings. Throughout the paper, we work with simple graphs, even when not stated explicitly.

\paragraph{Graphons.} A graphon is a Lebesgue measurable function $W\colon [0,1)^2\to[0,1]$ that is
symmetric in the sense that $W(x,y)=W(y,x)$ for all $x,y\in[0,1)$. Every graph $G=(V,E)$, where $V=[n]$, can be represented
by a graphon $W_G$ as follows: for each $i,j\in[n]$, and for $x\in [(i-1)/n,i/n]$ and $y\in [(j-1)/n,j/n]$, we set
\[
W_G(x,y)=
  \begin{cases}
    1, & \text{if $(i,j)\in E$}, \\
    0, & \text{otherwise}.
  \end{cases}
\]
For two finite simple graphs $F$ and $G$, we define the {\it homomorphism
density} of $F$ in $G$ as
\[
t(F,G)=\frac{\hom(F,G)}{|V(G)|^{|V(F)|}},
\]
where $\hom(F,G)$ is the number of homomorphisms of $F$ into $G$, that is, adjacency preserving maps $V(F)\to V(G)$. For a finite simple graph $F$ and graphon $W$, let
\[
t(F,W)=\int_{[0,1)^n} \prod_{ij\in E(F)} W(x_i,x_j)\diff\lambda_n(x).
\]
This generalizes the case of finite graphs, in the sense that
\[
t(F,G)=t(F,W_G).
\]
For a graph $G$ and positive integer $k$, we denote by $G(k)$ the graph obtained from the $G$ by replacing each node $u$
by a set $V_u$ of $k$ non-adjacent nodes and connecting all nodes of $V_u$ to all nodes of $V_v$ if $uv \in E(G)$. It is not hard to see that $t(F,G(k))=t(F,G)$ for all graphs $F$. 

A sequence of graphs $(G_n\colon n\in\bbZ_+)$ is called {\it left-convergent} if
$(t(F,G_n)\colon n\in\bbZ_+)$ is a convergent sequence of real numbers for every
finite simple graph $F$. We say that the sequence $(G_n\colon n\in\bbZ_+)$ {\it converges to a
graphon $W$}, in notation $G_n\to W$, if $t(F,G_n)\to t(F,W)$ is a convergent
sequence of real numbers for every finite simple graph $F$. It is proved 
in~\cite{LSzlimit} and~\cite{BCLunique} that for every left-convergent sequence of
graphs there is an essentially unique graphon to which it converges, and every
graphon is the limit of a left-convergent graph sequence.

\paragraph{Cut distance.} 
There are several equivalent ways of defining convergence of dense graph sequences, of which we need another one.
Let $G$ and $H$ be two graphs of the same node set $V(G)=V(H)=V$. We define their {\it labeled cut distance} as
\[
d_\square(G,H)=\max_{S,T\subseteq V}\frac{|e_G(S,T)-e_H(S,T)|}{|V|^2},
\]
where $e_G(S,T)$ denotes the number of edges with one endpoint in $S$ and one endpoint in $T$. Note that if both endpoints are in $S\cap T$, then we count the edge twice.

Let $G$ and $H$ be two graphs, possibly with different number of nodes $|V(G)|=n$ and $|V(H)|=m$. Then $G(mt)$ and $H(nt)$ have the same number of nodes for any $t\in\bbZ_+$. Let $H^\beta$ be the image of $H(nt)$ under a bijection $\beta\colon V(H(nz))\to V(G(mt))$. The {\it unlabeled cut distance} is defined as
\[
\delta_\square(G,H)=\inf d_\square(G(mt),H(nt)^\beta),
\]
where the infimum is taken over all positive integers $t$ and all bijections $\beta$.

As proved in~\cite{BCLSV1}, a sequence $(G_n\colon n\in\bbZ_+)$ of graphs is dense-convergent if and only if it is a Cauchy sequence with respect to the unlabeled cut distance.

For two graphons $u$ and $W$, their {\it labeled cut distance} is defined by
\[
\|U-W\|_\square=\sup_{S,T} \left|\int_{S\times T} (U(x,y)-W(x,y))\diff x\diff y\right|,
\]
where $S,T\subseteq[0,1]$ range over measurable sets. their {\it labeled cut distance} is defined by
\[
\delta_\square=\inf_\varphi \|U-W^\varphi\|_\square,
\]
where $\varphi$ ranges over all measure-preserving bijections $[0,1]\to[0,1]$, and $W^\varphi$ is the graphon defined by $W^\varphi(x,y)=W(\varphi(x),\varphi(y))$. By~\cite[Theorem 11.59]{lovasz2012large}, if $G_n\to W$, then we can label the nodes of the graphs $G_n$ so that for the sequence $W_{G_n}$ of representing graphons, $|W_{G_n}-W|_{\square}\to0$ holds as $n\to\infty$.

\subsection{Convergence in Hausdorff distance}
\label{sec:hausdorff}

Let $\varepsilon>0$ and $U$ be a subset of a compact metric space $K$. We denote the set of points being at distance at most $\varepsilon$ from $U$. It is easy to see that every increasing sequence of subsets of a compact metric space is convergent in the Hausdorff distance. The following lemma relaxes the increasing condition.

\begin{lem}\label{LEM:GROW-CONVERGE}
Let $(U_n\colon n\in\bbZ_+)$ be a sequence of nonempty subsets of a compact metric space $K$. Suppose that for every $\eps>0$, there exists an integer $f(\eps)\in\bbZ_+$ such that, for every $n\ge f(\eps)$, there exists an integer $f(n,\eps)$ with the property that for all $m\ge f(n,\eps)$, we have $U_n\subseteq U_m^\eps$. Then $(U_n\colon n\in\bbZ_+)$ is convergent in the Hausdorff distance.
\end{lem}

\begin{proof}
Let $\eps>0$. We may assume that $f(n,\eps)\ge n$. As long as it is possible, we select an increasing sequence $(n_q\colon q\in\bbZ_+)$ of positive
integers, starting with $n_0=f(\eps)$ with the following properties:
\[
n_{q+1}> f(n_q,\eps) \quad \text{and}\quad U_{n_{q+1}}\not\subseteq
\bigcup_{j=1}^{q} U_{n_j}^\eps.
\]
Then $U_{n_{q+1}}$ contains a point that is farther than $\eps$ from every
point of $U_{n_j}$ for all $j\le q$. Since $K$ is compact, this can happen only finitely many times, 
so the process terminates at some finite $q$. This means that for every $m>m_0=f(n_q,\eps)$, we have $U_m\subseteq
\bigcup_{j=1}^{q} U_{n_j}^\eps$. We have $U_{n_j}\subseteq U_{n_q}^\eps$ by the
definition of the sequence $(n_j\colon j\in\bbZ_+)$, and hence $U_{n_j}^\eps\subseteq
U_{n_q}^{2\eps}$. So it follows that $U_m\subseteq U_{n_q}^{2\eps}$ for every
$m>m_0$. On the other hand, $U_{n_q}\subseteq U_m^{\eps}$ for every $m>m_0$ by
the definition of $m_0$. Thus $d^\Haus(U_m,U_{n_q})\le 2\cdot \eps$ if $m>m_0$. This
implies that $d^\Haus(U_n,U_m)\le 4\cdot \eps$ if $n,m>m_0$, which proves the lemma.
\end{proof}

\section{Related convergence notions}
\label{sec:related}
\subsection{Crop-convergence}
\label{sec:crop}

We discuss some versions of quotient-convergence, which will provide a convenient step in proving the
more stringent quotient-convergence. Recalling for a setfunction on a set-algebra $(J,\cB)$ the definition of its $k$-quotient set
\[
Q_k(\fg)=\{\fg/\bA\colon \bA\in\cB^k,~\text{$\bA$~is a partition of $J$}\},
\]
we see that these are reasonably natural:
\begin{align*}
T_k(\fg)&=\{\fg/\bA\colon \bA\in\cB^k\}.\\
T_k^\Delta(\fg)&=\{\fg/\bA\colon \bA\in\cB^k,~\bA\text{~consists of disjoint sets}\},\\
T_k^\nabla(\fg)&=\{\fg/\bA\colon \bA\in\cB^k,~\bA\text{ covers }J\}.
\end{align*}
Clearly, $T_k(\fg)$ and its variants are sets of functions on $2^{[k]}$. The setfunctions in $T_k(\fg)$ are called the {\it $k$-crops} of $\fg$. Crops play an important role in combinatorial optimization. For example, crops of finitely additive measures are called {\it coverage functions} (extending a terminology from the finite case), and can be characterized by the ``infinitely alternating'' condition of Choquet~\cite{choquet1954theory}. We call a sequence $(\fg_n\colon n\in\bbZ_+)$ of setfunctions {\it crop-convergent} if the corresponding sequence of sets $(T_k(\fg_n)\colon n\in\bbZ_+)$ is Cauchy with respect to the Hausdorff distance for every $k\ge 1$, and we say
that $(\fg_n\colon n\in\bbZ_+)$ {\it crop-converges} to a setfunction $\fg$, if
$T_k(\fg_n)\hto T_k(\fg)$ for every $k\ge 1$. We define {\it $\Delta$-convergence}
and {\it $\nabla$-convergence} analogously.

It is easy to verify that if $\fg$ is a quotient of $\psi$, then
$T_k(\fg)\subseteq T_k(\psi)$, $T_k^\Delta(\fg)\subseteq T_k^\Delta(\psi)$ and
$T_k^\nabla(\fg)\subseteq T_k^\nabla(\psi)$. It is also clear that for every
setfunction $\fg$,
\begin{equation*}\label{EQ:S-SDELTA}
Q_k(\fg) \subseteq T_k^\Delta(\fg)\subseteq T_k(\fg)
\qquad\text{and}\qquad Q_k(\fg) \subseteq T_k^\nabla(\fg)\subseteq T_k(\fg).
\end{equation*}
When the setfunction $\fg$ is clear from the context, let us simply write $Q_k\subseteq T_k^\Delta\subseteq T_k$ and $Q_k\subseteq
T_k^\nabla\subseteq T_k$. We also write 
\[
(T_k\circ T_m)(\fg)=\bigcup_{\psi\in T_m(\fg)} T_k(\psi) = \big\{(\fg/\bU)/\bV\colon \bU\in\cB^m, \bV\in(2^{[m]})^k\big\},
\]
and similarly for the other variants.

The crop profile is in several respects simpler than the quotient profile. For
example, the sets $T_k$ and $T_{k-1}$ are related in a very simple way:
setfunctions in $T_{k-1}$ can be obtained by restricting setfunctions in $T_k$
to subsets of $[k-1]$. We also have $Q_k\subseteq T_k$, but the sets $T_k$ will not approximate $Q_k$ well; for example, $\psi([k])=\fg(J)$ for every $\psi\in Q_k(\fg)$, but the identically zero function belongs to $T_k(\fg)$. The next simple lemma does establish some connection.

\begin{lem}\label{LEM:SK-SK-QK}
Let $a\ge k\ge 1$ and $b\ge 2^k$. Then $T_k=Q_k\circ T_a=T_k\circ T_a=T_k\circ
Q_b$.
\end{lem}

\begin{proof}
It is clear that
\[
Q_k\circ T_a \subseteq T_k\circ T_a \subseteq T_k, \qquad\text{and}\qquad  T_k\circ Q_b \subseteq T_k.
\]
So it suffices to verify that
\[
T_k\subseteq Q_k\circ T_a, \qquad\text{and}\qquad  T_k\subseteq T_k\circ Q_b.
\]
Let $\fg/\bA$ be an element of $T_k(\fg)$, where $\bA=(A_1,\dots,A_k)$ with
$A_i\in\cB$ for all $i\in[k]$. We define $\bA'=(A_1,\dots,A_k,A_k,\dots,A_k)$, where $A_k$ is repeated
$a-k+1$ times, and $\cX =(1,\dots,k-1,\{k,\dots,a\})$. Then $\fg/\bA =
(\fg/\bA')/\cX  \in Q_k(T_a(\fg))$.

To prove the second containment, let $C_1,\dots,C_b$ be the atoms of the set
algebra $\Af(\bA)$ with an appropriate number of empty sets added, and let
$\bP=(C_1,\dots,C_b)$. For $i\in[k]$, we can write $A_i=\cup_{j\in X_i} C_j$
for an appropriate $X_i\subseteq [b]$. Let $\cX =\{X_1,\dots,X_k\}$, then
$\fg/\bA=(\fg/\bP)/\cX \in T_k(Q_{2^k}(\fg))$.
\end{proof}

Using the equality $T_k =T_k\circ Q_{2^k}$, we obtain the following.

\begin{cor}\label{COR:QK2SK}
If a sequence $(\fg_n\colon n\in\bbZ_+)$ of setfunctions is quotient-convergent,
then it is crop-convergent. 
\end{cor}

In particular, every tower is crop-convergent, where a {\it tower} is a sequence $(\fg_n\colon n\in\bbN)$ of setfunctions such that $\fg_n$ is a quotient of $\fg_{n+1}$ for every $n$. However, the reverse implication does not necessarily hold: crop-convergence does not imply quotient-convergence in general. 

Let $(\MM_n=(E_n,r_n)\colon n\in\bbZ_+$) be a
sequence of matroids, with normalized rank functions $\rho_n=r_n/r_n(E_n)$. We
say that the sequence $(\MM_n\colon n\in\bbZ_+)$ is {\it crop-convergent} if their
normalized rank functions form a crop-convergent sequence. In the case of rank functions of finite matroids, we may restrict the sets $A_i$ in the definition of crop-convergence to flats. In other words, if $\LL_n$
is the lattice of flats of $\MM_n$, then the sets $T_k(\rho_n)$ depend only on the
lattice $\LL_n$ and on the restriction of $\fg_n$ to $\LL_n$.

Let $\MM=(E,r)$ and $\MM'=(E',r')$ be two matroids. Set $c=r(E)$ and $c'=r'(E')$. Also
consider the normalized rank functions $\rho=r/c$ and $\rho'=r'/c'$. Let $\LL$
and $\LL'$ be the lattices of flats in $\MM$ and $\MM'$, respectively. A
map $\Lambda\colon \LL\to\LL'$ is a {\it lattice embedding} if it is injective and
preserves joins and meets.  A lattice embedding $\Lambda: ~\LL\to\LL'$ is
called {\it rank preserving} if $r'(\Lambda(Y))=r(Y)$ for every $Y\in\LL$, and it
is a {\it stretch embedding} if $r'(\Lambda(Y))=(c'/c)\cdot r(Y)$ for every
$Y\in\LL$, or equivalently, $\rho'(\Lambda(Y))=\rho(Y)$.

We need a couple of simple consequences of the existence of these embeddings.

\begin{lem}\label{LEM:SK-DIV}
Let $\MM=(E,r)$ and $\MM'=(E',r')$ be two matroids with $c=r(E)$ and $c'=r'(E')$, denote by $\rho=r/c$ and $\rho'=r'/c'$ their normalized rank functions, and let $\LL$
and $\LL'$ be the lattices of flats in $\MM$ and $\MM'$, respectively, respectively.
\begin{enumerate}[label=(\alph*)]\itemsep0em
    \item If there exists a stretch embedding $\Lambda\colon \LL\to\LL'$, then
$T_k(\rho)\subseteq T_k(\rho')$. \label{it:a}
    \item If there exists a rank preserving embedding $\Gamma\colon \LL\to\LL'$, then $T_k(\rho)\subseteq (c'/c)\cdot T_k(\rho')$. \label{it:b}
\end{enumerate}
\end{lem}
\begin{proof}
First we prove~\ref{it:a}. Let $\psi\in T_k(\rho)$, then there exists a family
$ \bA=(A_1,\dots,A_k)\in\LL^k$ of flats of $\MM$ such that $\psi=\rho/ \bA$. Let
$B_i=\Lambda(A_i)$ and $ \bB=(B_1,\dots,B_k)\in(\LL')^k$. For $X\subseteq[k]$,
set $A_X=\cup_{i\in X} A_i$ and $B_X=\cup_{i\in X} B_i$. Let $\overline{A}_X$
denote the flat spanned by $A_X$. Then $\Lambda(\overline{A}_X)$ is the flat
spanned by $B_X$, and so we have
\[
\rho'(B_X) = \rho(\Lambda(\overline{A}_X))= \rho(\overline{A}_X)= \rho(A_X).
\]
This shows that $\psi=\rho/ \bA=\rho'/ \bB\in T_k(\rho')$. 

The proof of~\ref{it:b} is analogous.
\end{proof}

\begin{thm}\label{THM:MAT-CONV}
Let $(\MM_n=(E_n,r_n)\colon n\in\bbZ_+$) be a sequence of matroids with $r_n(E_n)=n$, and
normalized rank functions $\rho_n=r_n/n$. Let $\LL_n$ be the family of flats in
$\MM_n$. Assume that
\begin{enumerate}[label=(\roman*)]\itemsep0em
    \item for all $m\le n$, there exists a rank preserving embedding $\LL_m\to\LL_n$, and \label{it:i}
    \item if $m\mid n$, then there exists a stretch embedding $\LL_m\to\LL_n$. \label{it:ii}
\end{enumerate}
Then $(\rho_n\colon n\in\bbZ_+)$ is crop-convergent, and  the limit of
$k$-crops is given by $T_k(\rho_n)\hto \bT_k=\bigcup_{n\in\bbZ_+} T_k(\rho_n)$ for all $k\in\bbZ_+$.
\end{thm}
Condition~\ref{it:i} is equivalent to saying that for $m\leq n$, $\MM_m$ is the restriction of $\MM_n$ to a flat.
\begin{proof}
Suppose that $(T_k(\rho_n)\colon n\in\bbZ_+)$ does not converge to $\bT_k$. Then there exists an
$\eps>0$ and an infinite sequence $(n_i\colon i\in\bbZ_+)$ of natural numbers such that
$d^\Haus(T_k(\rho_{n_i}), \bT_k)\ge\eps$ for all $i\in\bbZ_+$. Since $T_k(\rho_{n_i})\subseteq \bT_k$,
this means that there are setfunctions $\psi_i\in \bT_k$ for $i\in\bbZ_+$ such that
$\|\psi_i-\sigma\|\ge\eps$ for every $\sigma\in T_k(\rho_{n_i})$. Since $ \bT_k$
is bounded, we may assume that the setfunctions $\psi_i$ form a Cauchy
sequence. Let $s\in\bbZ_+$ be chosen so that $\|\psi_i-\psi_j\|\le\eps/2$ for
all $i,j\ge s$.

Since $\psi_s\in \bT_k$, there is an $m$ such that $\psi_s\in \bT_k(\rho_m)$.
Write $n_i=a_i\cdot m+b_i$, where $a_i,b_i\ge 0$ are integers and $b_i<m$. By Lemma
\ref{LEM:SK-DIV}\ref{it:a}, we have $\psi_s\in T_k(\rho_{a_im})$; by Lemma
\ref{LEM:SK-DIV}\ref{it:b}, it follows that
\[
\sigma_i=\frac{a_i\cdot m}{n_i} \psi_s\in T_k(\rho_{n_i}),
\]
and
\[
\|\psi_s-\sigma_i\| = \Big\|\Big(1-\frac{a_i\cdot m}{n_i}\Big)\cdot\psi_s\Big\| \le \frac{b_i}{n_i}< \frac{m}{n_i}<\frac{\eps}{2}
\]
if $i$ is large enough. Then
\[
\|\psi_i-\sigma_i\|\le \|\psi_i-\psi_s\| + \|\psi_s-\sigma_i\| <\frac{\eps}{2}+\frac{\eps}{2} =\eps,
\]
contradicting the definition of $\psi_i$.
\end{proof}

\begin{cor}\label{COR:MAT-CONV}
Let $(\MM_n=(E_n,r_n)\colon n\in\bbZ_+)$ be a sequence of matroids, with
$r_n(E_n)=n$ and normalized rank functions $\rho_n=r_n/n$. Assume that
$\MM_a\oplus\MM_b$ embeds into $\MM_{a+b}$ for every $a,b\ge 1$. Then the
sequence $(\MM_n\colon n\in\bbZ_+)$ is crop-convergent.
\end{cor}
\begin{proof}
We show that conditions~\ref{it:i} and~\ref{it:ii} of Theorem~\ref{THM:MAT-CONV} hold. We
may assume that $m<n$. First, $\MM_m$ trivially embeds into $\MM_m\oplus
\MM_{n-m}$, and in turn, this embeds into $\MM_n$, so $\MM_m$ embeds into
$\MM_n$. Second, $\MM_m$ has a stretch embedding into $\MM_m\oplus\dots\oplus
\MM_m$ where we have $n/m$ terms, simply by $x\mapsto(x,\dots,x)$. In turn,
$\MM_m\oplus\dots\oplus \MM_m$ embeds into $\MM_n$; this follows by induction
from the hypothesis. So $\MM_m$ has a stretch embedding into $\MM_n$.
\end{proof}

\begin{ex}\label{EXA:COMPLETE-SCONV}
Let $\MM_n$ be the cycle matroid of the complete graph $K_{n+1}$. Let $a,b\ge
1$, and let $G$ denote the graph obtained by identifying a node of $K_{a+1}$
with a node of $K_{b+1}$. It is easy to see that the cycle matroid $\MM_G$ of
$G$ is isomorphic to $\MM_a\oplus\MM_b$, and trivially $\MM_G$ is a
restriction of $\MM_{a+b}$. So $(\MM_n\colon n\in\bbZ_+)$ is crop-convergent by
Corollary~\ref{COR:MAT-CONV}.
\end{ex}

\begin{ex}\label{EXA:LINSPACE-SCONV}
Let $\bbF$ be a finite field, $|\bbF|=q$, and let $\rho_n=\dim/n$ denote the
normalized rank function of the linear space $\bbF^n$. It is easy to see that the
matroid direct sum of $\bbF^a$ and $\bbF^b$ ``almost'' embeds into $\bbF^{a+b}$ -- the $0$ vectors in $\bbF^a$ and $\bbF^b$ must be identified. It is not difficult to check that this is enough for conditions~\ref{it:i} and~\ref{it:ii} of Theorem~\ref{THM:MAT-CONV} to hold. Thus,
the sequence $(\rho_n\colon n\in\bbZ_+)$ is crop-convergent. In Section~\ref{sec:linear}, we will show that it is also quotient-convergent.
\end{ex}

\subsection{Quotient-convergence versus crop-convergence}
\label{sec:quotientvscrop}

We need to introduce a property expressing that a matroid is ``rich'':
\begin{equation*}\leqnomode
    \text{For any two flats $F\subseteq A$ where $r(A)\ge m$, we have
$|A\setminus F|\ge k\cdot (r(A)-r(F))$.} \tag*{(R(k,m))}\linkdest{rkm}{} \label{eq:rkm}
\end{equation*}
In particular, this implies that for any flat $A$ with $r(A)\ge m$, we have $|A|\ge k\cdot r(A)$.

\begin{lem}\label{LEM:M-A-BASE}
Let $\MM=(E,r)$ be a matroid satisfying~\ref{eq:rkm} for some $k,m\ge 1$.
Let $A_1,\dots,A_k$ be flats of $\MM$ with $r(A_i)\ge m$ for every $i\in[k]$. Then
there are disjoint subsets $B_i\subseteq A_i$ such that $B_i$ is a basis of
$A_i$ for $i\in[k]$, that is, $B_i$ is an independent subset of $A_i$ of size $r(A_i)$.
\end{lem}
\begin{proof}
Set $r_i(X)\coloneqq r(X\cap A_i)$ for all $X\subseteq E$. Then $\MM_i=(E,r_i)$ is a matroid, and we need to show that there are disjoint bases $B_i$ of $\MM_i$. This is equivalent to the
assertion that the sum $\MM_\Sigma$ of the matroids $\MM_1,\dots,\MM_k$ has rank at least 
$\sum_{i=1}^k r(A_i)$. Let $r_\Sigma$ denote the rank function of $\MM_\Sigma$. Then, by the Matroid Sum Theorem~\cite{edmonds1965transversals}, we have
\[
r_\Sigma(E) = \min\Bigl\{|Y| + \sum_{i=1}^k r_i(E\setminus Y)\colon Y\subseteq E\Bigr\}.
\]
Therefore, it suffices to show that 
\begin{equation*}\label{EQ:Y-BIG}
|Y| + \sum_{i=1}^k r_i(E\setminus Y)
=|Y| + \sum_{i=1}^k r(A_i\setminus Y)
\ge \sum_{i=1}^k r(A_i)
\end{equation*}
for all $Y\subseteq E$. Let $F_i$ be the closure of $A_i\setminus Y$, and let
$i\in[k]$ be an index such that $r(A_i)-r(F_i)=\max\{r(A_j)-r(B_j)\colon j\in[k]\}$. Then, by~\ref{eq:rkm},
\[
|Y| \ge |A_i\cap Y| \ge |A_i\setminus F_i|\ge k \cdot (r(A_i)-r(F_i))\ge \sum_{i=1}^k (r(A_i)-r(F_i)),
\]
and so
\[
|Y|+ \sum_{i=1}^k r(A_i\setminus Y) = |Y|+ \sum_{i=1}^k r(F_i) \ge \sum_{i=1}^k r(A_i),
\]
concluding the proof of the lemma.
\end{proof}

\begin{lem}\label{LEM:M-A-BASE2}
Let $\MM=(E,r)$ be a matroid satisfying~\ref{eq:rkm} for some $m\ge k\ge 1$.
Then
\[
d^\Haus\bigl(T_k(\rho),T_k^\Delta(\rho)\bigr)\le \frac{k\cdot m}{r(E)}
\]
and
\[
d^\Haus\bigl(T_k^\nabla(\rho),Q_k(\rho)\bigr)\le \frac{k\cdot m}{r(E)}.
\]
\end{lem}

\begin{proof}
It suffices to show that every $\psi\in T_k(\rho)$ can be approximated by a $\psi'\in T_k^\Delta(\rho)$, and every $\sigma\in T_k^\nabla(\rho)$ can be approximated by a $\sigma'\in Q_k(\rho)$.

By definition, $\psi=\rho/ \bA$ for some $\bA=(A_1,\dots,A_k)\in  \bB^k$. Since taking the closure of each $A_i$ does not change $\psi$, we may assume that the sets $A_i$ are flats. For every $i$ with $r(A_i)<m$, we increase $A_i$ to a
flat $A_i'$ with $r(A_i')=m$; else we let $A_i'=A_i$. Let
$ \bA'=(A'_1,\dots,A'_k)$. By Lemma~\ref{LEM:M-A-BASE}, there exist disjoint sets
$B_i$ such that $B_i$ is a basis of $A_i'$ for $i\in[k]$. Let
$ \bB=(B_1,\dots,B_k)$, then $\psi'=\rho/ \bB=\rho/ \bA'\in T_k^\Delta(\rho)$. It
is clear that $\psi'\ge\psi$. Moreover, subadditivity of the rank function
implies that
\[
\psi'(X)=\rho\Bigl(\bigcup_{i\in X} A'_i\Bigr)\le \rho\Bigl(\bigcup_{i\in X} A_i\Bigr) + \sum_i(\rho(A_i')-\rho(A_i))
\le \psi(X)+\frac{k\cdot m}{r(E)},
\]
proving the first inequality.

To prove the second, let $\sigma=\rho/ \bA$, where $ \bA=(A_1,\dots,A_k)\in \bB^k$
covers $E$. We construct the sets $A_i'$ and $B_i$ as above, and let $B=\bigcup_{i=1}^{k}
B_i$. We add each element $u\in E\setminus B$ to $B_i$ where $i$ is the smallest index
with $u\in A_i$. Since $\bigcup_{i=1}^{k}A_i=E$, this gives a partition $\cP$ of $E$ for which $\psi'=\rho/\cP$ is close to $\sigma$ by the same calculation as
above.
\end{proof}

As a corollary, we get the following.

\begin{cor}\label{COR:MISC-CONVERGE}
Let $f\colon\bbZ_+\to\bbZ_+$ be a function, and let $(\MM_n=(E_n,r_n)\colon n\in\bbZ_+)$ be a
sequence of matroids satisfying $r_n(E_n)\to \infty$ and \hyperlink{rkm}{($R(k,f(k))$)} for every $k\in\bbZ$.
\begin{enumerate}[label=(\alph*)]\itemsep0em
    \item The sequence $(r_n\colon n\in\bbZ_+)$ is crop-convergent if and only if it is $\Delta$-convergent.
    \item The sequence $(r_n\colon n\in\bbZ_+)$ is quotient-convergent if and only if it is $\nabla$-convergent.
\end{enumerate}
\end{cor}

\section{Finite linear spaces}
\label{sec:linear}

We start with a special example to illustrate the use of crop-convergence in
proving quotient-convergence.

\begin{thm}\label{THM:FIN-GEOM}
Let $ \bbF$ be a finite field, $|\bbF|=q$, and let $\rho_n=\dim/n$ denote the
normalized rank function of the linear space $\bbF^n$. Then the sequence
$(\rho_n\colon n\in\bbZ_+)$ is quotient-convergent.
\end{thm}
\begin{proof}
We start by observing that the sequence $(\bbF^n\colon n\in\bbZ_+)$ has two important properties. First, if $m\le n$, then the map $\LL_m\to\LL_n$ defined by $A\mapsto A\oplus \mathbf{0}$, where $\mathbf{0}$ is the zero subspace of $\bbF^{n-m}$, is a rank-preserving lattice embedding. Second, if $m\mid n$, then the map $\LL_m\to\LL_n$ defined by $A\mapsto A\oplus\dots\oplus A$, with $n/m$ terms, is a stretch embedding. 

It follows that the sequence $(\bbF^n\colon n\in\bbZ_+)$ satisfies the conditions of Theorem~\ref{THM:MAT-CONV}, and the sequence $(T_k(\bbF^n)\colon n\in\bbZ_+)$ is
convergent in the Hausdorff distance; let $ \bT_k$ be its limit. It is not hard to verify that every matroid $\bbF^n$ satisfies condition
\hyperlink{rkm}{($R(k,2k)$)} for every $k$ and $n$, and hence Lemma~\ref{LEM:M-A-BASE2} implies
that the sequence $(T^\Delta_k(\bbF^n)\colon n\in\bbZ_+)$ converges to $ \bT_k$ in the
Hausdorff distance as well. We claim that the sets $Q_k(\rho_n)$ converge in Hausdorff distance to the set
\[
 \bQ_k=\{\psi\in \bT_k\colon \max_i\psi(\{i\})=1\}
\]
for every $k\in\bbZ_+$. Let $\psi\in Q_k(\rho_n)$, and write $\psi=\rho_n/\cP$ where $\cP=(A_1,\dots,A_k)$ is a measurable partition of $ \bbF^n$. Clearly,
there exists $i\in[k]$ such that $|A_i|\ge q^n/k$. This implies that
$\dim(A_i)\ge n-\log_q k$ and hence
\begin{equation*}\label{EQ:QQK1}
\max_{i\in[k]} \psi(\{i\}) \ge  1-\frac{\log_q k}{n}.
\end{equation*}
This, together with $Q_k(\rho_n)\subseteq T_k(\rho_n)$, implies that every limit point of the sequence
$Q_k(\rho_n)$ is in $\bQ_k$. On the other hand, let
$\psi\in \bQ_k$, where we may assume without loss of generality that $\psi(\{k\})=1$. This implies that $\psi(X)=1$ for every $X\subseteq [k]$ such that $k\in X$.

Let $\eps>0$. We know that $\psi \in \bT_k$, so there exists a family
$ \bB=(B_1,\dots,B_k)$ of disjoint subsets of $\bbF^n$ such that
$\|\rho_n/ \bB-\psi\|\le \eps$. Let us replace $B_k$ by the set
$B_k'=\bbF^n\setminus\bigcup_{j=1}^{k-1} B_j$, then we get a partition $\cP$ of
$\bbF^n$ such that $\rho_n/\cP\ge\rho_n/ \bB$, and so
\[
\|\rho_n/\cP-\rho_n/ \bB\|=\rho_n(B'_k)-\rho_n(B_k) \le \frac{\log_q k}{n}.
\]
It follows that 
\[
\|\rho_n/\cP-\psi\| \le \eps+\frac{\log_q k}{n} =\eps+o(1),
\]
proving that $Q_k(\bbF^n)\hto \bT_k$ as $n\to\infty$.
\end{proof}

Despite the above progress, we cannot yet describe the limit object in precise terms.

\begin{prob}\label{PROB:NEUMANN}
What is a good representation of the limit of this sequence of finite linear
spaces? Is it related to von Neumann's continuous geometry over $\bbF$?
\end{prob}

\section{Dense convergence and quotient-convergence}
\label{sec:dense}
\subsection{Cycle matroids}
\label{sec:cycle}

For a graph $G=(V,E)$, we denote the rank function of its cycle matroid by $r_G\colon 2^E\to \bbZ_+$. We normalize by $|V|$, to get $\rho_G=r_G/|V|$. Note that this is slightly different from the normalization of general matroids, but for the graph sequences we will be interested in, it is asymptotically the same.

In~\cite{Lov23b}, the matroid of a graphing was defined, and in~\cite{BBLTgraphing} it was proved that for a locally-globally convergent sequences of bounded-degree graphs, their cycle matroids quotient-converge to the matroid of the limit graphing. It is natural to conjecture that the cycle matroids of a dense-convergent sequence of graphs also form a quotient-convergent sequence. However, this is not the case; see Example~\ref{EXA:DENSE-NOT-CONV} for a counterexample. On the other hand, we do establish convergence for the sequence of graphs which converge to a positive graphon.

\begin{ex}\label{EXA:DENSE-NOT-CONV}
Consider the following sequence of graphs: let $G_1 = K_1$, $G_2 = K_2$, and for $n \ge 3$, define $G_n$ to be a tree if $n$ is odd, and the edge-disjoint union of two spanning trees if $n$ is even, with $|V(G_n)| = n$ for all $n \in \mathbb{Z}_+$. The sequence $(G_n\colon n\in\bbZ_+)$ is
convergent in the dense sense as it tends to the zero graphon. On the other
hand, the sequence $(\rho_n\colon n\in\bbZ_+)$ of their normalized rank functions is
not quotient-convergent. For odd $n$, every $2$-quotient $\fg$ in $Q_2(\rho_n)$
satisfies $\fg(\{1\})+\fg(\{2\})=\fg(\{1,2\})=1$, while for even $n$, there is a
$2$-quotient $\fg$ with $\fg(\{1\})=\fg(\{2\})=\fg(\{1,2\})=1$, at distance $1/2$ from every
$Q_2(\rho_{n'})$ where $n'$ is odd. 

Note that the graphs in the sequence thus obtained are not dense since $G_n$ has $o(n^2)$ edges. However, one can modify the construction to obtain a sequence of dense graphs by considering the node-disjoint union of $G_n$ and the complete graph $K_n$ for $n\in\bbZ_+$. For every even $n\ge 4$, $E(G_n)$ has a partition into 4 classes such that two of those are spanning trees of $G_n$, while the other two form connected spanning subgraphs of $K_n$. The corresponding quotient $(U,\rho)$ is the matroid rank function of two pairs of parallel points, divided by $2$. It is easy to check that for large odd $n$, no such partition exists even $\varepsilon$-close to this.
\end{ex}

We say that a graphon $W$ is positive if $W(x,y)>0$ for almost every $x,y\in[0,1]$, denoted by $W>0$. Our main result in this section is the following.

\begin{thm}\label{THM:DENSE_CYCLIC_CONV}
    Let $(G_n\colon n\in\bbZ_+)$ be a sequence of dense graphs that converges to a positive graphon $W$. Then the sequence $(\rho_n\colon n\in\bbZ_+)$ of normalized cycle matroid rank functions is quotient-convergent. Furthermore, the limit does not depend on $W$.
\end{thm}

The rest of the section is devoted to proving the the theorem. 

\subsubsection{Graphs close to a positive graphon}

We will need a well-known fact from analysis.

\begin{prop}\label{PROP:INT_OF_POS_FUNC}
    Let $(\Omega,\mu)$ a probability space, and let $f\colon \Omega\to \R$ be a measurable function such that $\mu\{f>0\}=1$. Then for all $\varepsilon>0$, there exists a $\delta>0$ such that if $\mu(H)>\varepsilon$, then $\int_Hf\diff\mu>\delta$.
\end{prop}

We need some lemmas about graph-theoretic consequences for a graph to be close to a positive graphon. In the rest of the section, we consider $W$ to be a fixed positive graphon. We start with a simple observation.

\begin{lem}\label{LEM:GRAPH-BIG}
For every $n\in\bbN$, there exists $\delta>0$ such that, if $G=(V,E)$ is a graph with $\|W_G-W\|_\square<\delta$, then $|V|> n$.
\end{lem}
\begin{proof}
    As $G$ is simple, $W_G$ is $0$ on the diagonal. Therefore $\|W_G-W\|_{\square}>0$.
    There is a finite number of graphs with at most $n$ nodes, and each has  positive distance from $W$, since the corresponding graphons are 0 blocks on the diagonal. We can choose $\delta$ to be the minimum of these distances. 
\end{proof}

Let $\lambda_n$ denote the Lebesgue measure in $\bbR^n$. For any graph $F$, the function $W^F\colon~ (x_1,\dots,x_r)\mapsto \prod_{(i,j)\in E(F)} W(x_i,x_j)$ is almost everywhere positive. Hence for every $\delta_0>0$, Proposition~\ref{PROP:INT_OF_POS_FUNC} yields the existence of a $\varepsilon>0$ such that for any $S\subseteq[0,1]^r$ with $\lambda_r(S)\ge \varepsilon$, we have 
    \begin{equation}
    \int_{S}\prod_{ij\in E(K_r)}W(x_i,x_j)\diff \lambda_r(x)\ge \delta_0. \label{eq:d0}
    \end{equation}
    
We will use this lower bound in the following proofs.

\begin{lem}\label{LEM:COVER-BY-SMALL-Kk}
    For every $r\in\bbZ_+$ and $\varepsilon>0$, there exists $\delta>0$ such that, if $G=(V,E)$ is a graph with $\|W_G-W\|_\square<\delta$ and $A\subseteq V$ with $|A|\ge \varepsilon\cdot |V|$, then $G[A]$ contains a $K_r$. 
\end{lem}
\begin{proof} 
    Let $\delta_0>0$ be the value corresponding to $\varepsilon$ provided by~\eqref{eq:d0}, and set $\delta\coloneqq \delta_0/(2\cdot \binom{k}{2})$. Consider any graph $G$ with $\|W_G-W\|_{\square}<\delta$. Let $A\subseteq V(G)$ satisfying $|A|\ge\varepsilon\cdot |V(G)|$, and let $A'\subseteq[0,1]$ be the set corresponding to the nodes in $A$. 
    We define a graphon $W'$ as
    \[
    W'(x,y)=\begin{cases}
        W(x,y) &\text{if $x,y\in A'$},\\
        0& \text{otherwise}.
    \end{cases}
    \]
    Let $G'$ denote the graph obtained from $G$ by deleting all edges not spanned by $A$. Clearly, $\|W_{G'}-W'\|_\square\le \|W_{G}-W\|_{\square}< \delta$. By the Counting Lemma~\cite[Lemma 10.23]{lovasz2012large}, 
    $|t(K_r,G')-t(K_r,W')|\le \binom{k}{2}\cdot \|W_{G'}-W'\|_\square\le \delta_0/2.$
     Since~\eqref{eq:d0} implies $t(K_r,W')\geq\delta_0$ and hence $t(K_r,G')\ge\delta_0/2>0$, $G'$ contains a complete graph $K_r$.
\end{proof}

\begin{cor}\label{COR:COVER-BY-SMALL-Kk}
For every $r\in\bbZ_+$ and $\varepsilon>0$, there exists $\delta>0$ such that, if $G=(V,E)$ is a graph with $\|W_G-W\|_\square<\delta$, then there exist node-disjoint copies of $K_r$ in $G$ covering at least $(1-\varepsilon)\cdot|V(G)|$ nodes.
\end{cor}
\begin{proof}
Take node-disjoint copies of $K_r$ greedily; by Lemma \ref{LEM:COVER-BY-SMALL-Kk}, fewer than $\varepsilon\cdot |V|$ nodes remain uncovered.
\end{proof}

Given a graph $G=(V,E)$ and $A,B\subseteq V$, we denote the number of edges with one endpoint in $A$, and the other in $B$, by $e(A,B)$. 

\begin{lem}\label{LEM:NON-COMP-CLASSES-0}
        For every $\varepsilon>0$, there exists $\delta>0$ such that, if $G=(V,E)$ is a graph with $\|W_G-W\|_\square<\delta$ and $A,B\subseteq V$ with $|A|,|B|\ge \varepsilon\cdot |V|$, then $e(A,B)\ge \delta\cdot |V|^2$.
\end{lem}
\begin{proof}
    By Proposition~\ref{PROP:INT_OF_POS_FUNC}, there exists a $\delta>0$ such that $\int_{X\times Y}W\diff \lambda_2\ge 3\,\delta$ whenever $\lambda_2(X\times Y)\ge \varepsilon^2$. Consider any graph $G$ with $\|W_G-W\|_{\square}<\delta$. Let $A,B\subseteq V(G)$ satisfying $|A|,|B|\ge\varepsilon\cdot |V(G)|$, and let $A'\subseteq  [0,1]$ and $B'\subseteq[0,1]$ be the sets of points in $W_G$ corresponding to the nodes in $A$ and $B$, respectively. We define a graphon $W'$ as
    \[
    W'(x,y)=\begin{cases}
        W(x,y) &\text{if $(x,y)\in (A'\times B')\cup(B'\times A')$},\\
        0& \text{otherwise.}
    \end{cases}
    \]
Let $G'$ denote the subgraph of $G$ consisting of those edges having one endpoint in $A$ and the other one in $B$. Clearly, $\|W_{G'}-W'\|_\square\le \|W_{G}-W\|_{\square}\le\delta$. By the Counting Lemma~\cite[Lemma 10.23]{lovasz2012large}, we get $|t(K_2,W')-t(K_2,W_{G'})|\le \delta$. by the definition of $\delta$, we have $t(K_2,W')\ge 3\,\delta$, and so $e(A,B)\ge t(K_2,W_{G'})\cdot|V|^2/2\ge (t(K_2,W') - \delta)\cdot|V|^2/2 \ge \delta\,|V|^2$.
\end{proof}

\subsubsection{Edge-colorings} 

We need a couple of additional lemmas about graphs that are close to the positive graphon $W$, focusing on properties of their edge-colorings, which define their quotients. For an edge-coloring $c\colon E(G)\to[\ell]$ of a graph $G$, let $D_c(u)$ denote the number of nodes in the connected component of color $c$ containing $u\in V(G)$. For $A\subseteq[\ell]$, an {\it $A$-component of $c$} is a connected component of the subgraph formed by edges of color in $A$.

\begin{lem}\label{LEM:PAIRWASE-dense-colors}
     For every $\ell,N\in\mathbb{Z}_+$ and $\varepsilon>0$, there exists a number $\delta>0$ such that, if $G=(V,E)$ is a graph with $\|W_G-W\|_\square<\delta$, and $c\colon E(G)\to[\ell]$ is an edge-coloring of $G$, then there exists a set $U\subseteq V(G)$ with $|U|\ge(1-\varepsilon)\cdot |V|$ and a function $f\colon U\times U\to[\ell]$ such that for any two nodes $u,v\in U$ we have $D_{f(u,v)}(u),D_{f(u,v)}(v)\ge N$.
\end{lem}

\begin{proof} 
    For each node $v$, let $L_v$ be the set of colors $c$ for which the $c$-colored component of $v$ has size at least $N$. Note that if $d(v)> \ell\cdot (N-1)$, then $|L_v|\geq 1$. Consider the partition of the nodes into $2^\ell$ -- maybe empty -- classes, where two nodes $u$ and $v$ are in the same class if and only if $L_u=L_v$. We denote the class of $v$ by $P_v$. For a class $A$, we denote by $L_A$ the set of colors corresponding to the nodes in $A$.  We call two classes $A$ and $B$ compatible, denoted $A\sim B$, if they have a color in common. That is, $A\sim B$ if $L_A\cap L_B\neq\emptyset$. We show that if $A \not \sim B$, then at least one of these two classes is `small'.
    
    \begin{cl}\label{CL:NON-COMP-CLASSES}
        For every $\varepsilon'>0$, there exists $\delta>0$ such that if $A$ and $B$ are two non-compatible classes of $G$, then $\min\{|A|,|B|\}\leq \varepsilon'\cdot |V|$. 
            \end{cl}
    \begin{claimproof}
            Let $\delta_1$ be the value $\delta$ in Lemma \ref{LEM:NON-COMP-CLASSES-0} for $\varepsilon'$, and let $\delta_2$ be the value $\delta$ in Lemma  \ref{LEM:GRAPH-BIG} for $n=N\cdot \ell/\delta_1$.  Let $\delta=\min\{\delta_1,\delta_2\}$.
            Assume that there exist incompatible classes $A \not\sim B$ such that $|A|,|B|\ge\varepsilon' \cdot|V|$. By the choice of $\delta$, we have $|V|>N\cdot \ell/\delta_1$ and $e(A,B)\ge \delta_2\cdot |V|^2$.
            We obtain a contradiction by bounding the number of edges between $A$ and $B$ locally, using the fact that they have different color lists.
        
            Observe that for every $a\in A$, the number of edges from $a$ to $B$, denoted $\deg(a,B)$, is less than $\ell \cdot N$. Indeed, if $c_0\in L_a$, then no edge of color $c_0$ can connect $a$ to any node $b\in B$, since then $b$ would belong to a large $c_0$-component, implying $c_0\in L_b$ and thus contradicting $A\not\sim B$. On the other hand, if $c_0\notin L_a$, then $a$ has fewer than $N$ edges of color $c_0$ by definition. Thus, $\deg(a,B)\le |L_a|\cdot (N-1)<\ell\cdot  N$. So the number of edges between $A$ and $B$ is at most
            $|A|\cdot\ell\cdot N\le |V|\cdot\ell\cdot N$. Hence $|V|\cdot\ell\cdot N \ge\delta_1\cdot |V|^2$, or equivalently, $|V|\le \ell\cdot N/\delta_1$, a contradiction.
            \end{claimproof}

    Let $\varepsilon'=\varepsilon/\binom{2^\ell}{2}$. By Claim~\ref{CL:NON-COMP-CLASSES}, for each non-compatible pair, at least one of the classes has size at most $\varepsilon'\cdot|V(G)|$. Let $U\subseteq V(G_n)$ be the set of nodes not contained in small classes, or in other words, get rid of the small classes. Then $|U|\geq (1-\binom{2^\ell}{2}\cdot \varepsilon')\cdot|V(G_n)|=(1-\varepsilon)\cdot|V(G_n)|$. Furthermore, for any pair $u,v\in U$, their classes are compatible, so we can choose $f(u,v)$ to be any color in $L_u\cap L_v$, concluding the proof of the lemma.
\end{proof}

\begin{lem}\label{LEM:DENSE_CYCLIC_CLOSE}
Let $W$ and $W'$ be two positive graphons. For all $\ell\in\bbZ_+$ and $\varepsilon>0$, there exists a $\delta_1>0$ such that for all graphs $G$ with $\|W_G-W\|_\square<\delta_1$, there exists a $\delta_2>0$ such that for any graph $H$ with $\|W_H-W'\|_\square<\delta_2$, we have $Q_\ell(\rho_{G})\subseteq Q_\ell(\rho_{H})^\eps$.
\end{lem}

Given a graph $G=(V,E)$ and an edge-coloring $c\colon E\to[\ell]$, we denote by $\rho_G/c$ the quotient of the normalized rank function of $G$ corresponding to the partition into color classes. Stating the conclusion of the lemma in words: for every edge-coloring $c\colon E(G)\to[\ell]$, there exists an edge-coloring $f\colon E(H)\to[\ell]$ such that the quotients $\rho_G/c$ and $\rho_H/f$ are $\varepsilon$-close in the $\ell^{\infty}$ distance.

\begin{proof}
Let $G$ and $H$ be two graphs such that $\|W_G-W\|_\square<\delta_1$ and $\|W_H-W'\|_\square<\delta_2$, where $\delta_1>0$ and $\delta_2>0$  will be chosen small enough. Let $n=|V(G)|$ and $m=|V(H)|$. Note that $\delta_2$ may depend on $G$, so by choosing it small enough, we can makes sure that  $m>n/\varepsilon$ by Lemma \ref{LEM:GRAPH-BIG}.

Let $c\colon  E(G)\to[\ell]$ be an edge-coloring of $G$. We apply Lemma \ref{LEM:PAIRWASE-dense-colors} to the graphon $W$ and the graph $G$, with the given $\ell$, with $N=\lceil 4/\varepsilon\rceil$, and $\varepsilon/4$ in place of $\varepsilon$. If $\delta_1$ is small enough, then $G$ contains a set $U\subseteq V$ such that $|U|\ge(1-\varepsilon/4)\cdot n$, and there is a function $f\colon U\times U\to[\ell]$ such that $D_{f(u,v)}(u),D_{f(u,v)}(v)\ge N$ for all $u,v\in U$. Set $R=V\setminus U$.
    
By Corollary~\ref{COR:COVER-BY-SMALL-Kk}, if $\delta_2$ is small enough, then there exist node-disjoint copies of $K_n$ in $H$ covering $(1-\varepsilon/4)\cdot m$ nodes. In each of these complete graphs, we specify a copy of $G$; we denote these subgraphs by $G_1,\dots,G_q$. 
Let $G'=G_1\cup\dots\cup G_q$ and $S=V(H)\setminus V(G')$, and set $H'$ to be the subgraph of $H$ obtained by adding $S$ to $G'$ as isolated nodes. For every $u\in V(G)$ and $j\in[q]$, We denote by $(u,j)$ the node in $G_j$ corresponding to $u$. Set $U_j=\{(u,j)\colon u\in U\}$, $R_j=V(G_j)\setminus U_j$, and $S'=R_1\cup\dots\cup R_q\cup S$. 

We define a natural coloring of the edges of $G'$: for each $G_i$, we use the coloring $c$ of $G$. Denote this coloring by $c'\colon E(G')\to [\ell]$. Trivially, $\rho_G/c=\rho_{G'}/{c'}$. Since $E(H')=E(G')$, $c'$ is also an edge-coloring of $H'$. We claim that $\rho_{G'}/{c'}$ and $\rho_{H'}/{c'}$ are $\varepsilon/4$-close. Indeed, these quotients differ in their normalization only, and so
    \[
    \rho_{H'}/{c'} = \frac{|V(G')|}{|V(H')|}\cdot\rho_{H'}/{c'}=\frac{n\cdot q}{m}\cdot\rho_{H'}/{c'}
    \]
Since $(1-\varepsilon/4)\cdot m \le n\cdot q\le m$, this implies that
    \[
    |\rho_{G'}/{c'}(A)-\rho_{H'}/{c'}(A)| =  \frac{m-n\cdot q}{m}\cdot\rho_{G'}/{c'} (A) \le \frac{\varepsilon}{4}.
    \]

    We want to extend the coloring $c'$ to a coloring $\tc$ of $E(H)$ so that the quotient $\rho_H/\tc$ remains close to $\rho_G/c$. This is where we use the function $f$ from Lemma \ref{LEM:PAIRWASE-dense-colors}. We will obey one rule: whenever we color an edge incident to a node $v\in U_j$ with color $b$, then $D_v(b)\ge N$ must hold. This is easy to achieve by coloring an edge $e$ depending on the type of its endpoints as follows:
    \begin{itemize}
        \item $\tc((u,j),(u',j'))=f(u,u')$ for all $u,u'\in U$, $j,j'\in [q]$;
        \item $\tc((u,j),(r,j'))=f(u,u)$ for all $u\in U,\ r\in R,\ j,j'\in [q]$;
        \item $\tc(e)\in[\ell]$ is arbitrary, if both endpoints of $e$ are in $S'$.
    \end{itemize}

\begin{cl}
    For every $A\subseteq[\ell]$, every $A$-component of $c'$ contained in $U\times[q]$ with fewer than $N$ nodes remains an $A$-component of $\tc$.
\end{cl}
\begin{claimproof}
    For every edge $(x,y)$ of $H$ leaving such a component, we must have $f(x,y)\notin A$, so the component does not increase. 
\end{claimproof}    
 
Now we show that $\rho_{H'}/{c'}$ and $\rho_H/\tc$ are close. Let $A\subseteq[\ell]$.
Clearly, $\rho_{H}/\tc(A)\ge \rho_{H'}/{c'}(A)$. On the other hand, observe that whenever an $A$-component of $c'$ has changed, either the component already contained a node from $S'$, or it had size at least $N$. So the number of such components is at most 
\[
r=|R|\cdot q+|S|+\frac{m}{N}\le \frac{\varepsilon}{4}\cdot  n\cdot q + (m-n\cdot q) + \frac{\varepsilon}{4}\cdot  m \le \frac34\cdot \varepsilon\cdot  m.
\]
Therefore, the rank of the set of $A$-colored edges increases by less than $\frac34\cdot \varepsilon\cdot  m$. Hence 
\[
\rho_H/\tc(A)\le \rho_{H'}/{c'}(A) +\frac34\cdot  \varepsilon\le \rho_G/c(A)+\eps.
\]
This concludes the proof of the lemma.
\end{proof}

\subsubsection{Proof of Theorem~\ref{THM:DENSE_CYCLIC_CONV}}

\begin{proof}[Proof of Theorem~\ref{THM:DENSE_CYCLIC_CONV}]
Fix $\ell\in\bbZ_+$. We want to verify that the sets $Q_\ell(\rho_{G_n})$ satisfy the condition of Lemma \ref{LEM:GROW-CONVERGE}. Let $\eps>0$. Lemma~\ref{LEM:DENSE_CYCLIC_CLOSE}, when applied for $W'=W$, yields a value $\delta_1$; let $f(\eps)$ be chosen so that $\|W_{G_n}-W\|_\square<\delta_1$ for $n\ge f(\eps)$. Furthermore, Lemma~\ref{LEM:DENSE_CYCLIC_CLOSE}, when applied for $G=G_n$, yields a value $\delta_2>0$ that depends on $G$. Let $f(n,\eps)$ be such that if $m\ge f(m,\eps)$, then $\|W_{G_m}-W\|_\square<\delta_2$. Then, for $m\ge f(m,\eps)$, we have
$Q_\ell(\rho_{G_n})\subseteq Q_\ell(\rho_{G_m})^\eps$. Thus Lemma \ref{LEM:GROW-CONVERGE} implies that $Q_\ell(\rho_{G_n})$ is convergent in the Hausdorff distance.

To prove that the limit does not depend on $W$, let $G_n\to W$ and $H_n\to W'$, where $W$ and $W'$ are positive graphons. By \cite[Theorem 11.59]{lovasz2012large}, we can label the nodes of $G_n$ and $H_m$ such that $\|W_{G_n}-W\|_{\square}\to0$ and $\|W_{H_m}-W'\|_{\square}\to0$.
We know that $\lim Q_\ell(\rho_{G_n})$ and $\lim Q_\ell(\rho_{H_n})$ exist. Assume that they are not equal; without loss of generality, we may assume that there exists a point $a\in\lim Q_\ell(\rho_{G_n})$ and an $\varepsilon>0$ such that $B(a,\varepsilon)\cap Q_\ell(\rho_{H_n})=\emptyset$ for all, except for finitely many $n$. This contradicts Lemma~\ref{LEM:DENSE_CYCLIC_CLOSE}, proving the theorem.
\end{proof}

\begin{rem}
Theorem~\ref{THM:DENSE_CYCLIC_CONV} is equivalent to the following: {\it If $(G_n\colon n\in\bbZ_+)$ is a sequence of dense graphs such that the limit of every convergent subsequence is a positive graphon, then the sequence $(\rho_n\colon n\in\bbZ_+)$ of normalized cycle matroid rank functions is quotient-convergent.} It would be interesting to describe the limit object in a transparent way.
\end{rem}

\subsection{Cut capacities}\label{SSEC:GRAPHON-CUT}
\label{sec:cut}

Let $G=(V,E)$ be a finite graph. We define the {\it normalized cut capacity function} $\kappa_G(X)=e(X,V\setminus X)/|E|$, where $e(X,V\setminus X)$ is the number of
edges going between $X$ and $V\setminus X$, or in other words, the size of the cut defined by $X\subseteq V$.

Let $J=[0,1)$ and $W\colon J\times J\to [0,1]$ be a graphon. Then the cut capacity function $\kappa_W$ of the graphon is then defined as
\[
\kappa_W(X) = \frac{\int_{X\times (J\setminus X)}W(x,y)\diff\lambda(x)\diff\lambda(y)}{\int_{J\times J}W(x,y)\diff\lambda(x)\diff\lambda(y)}
\]
The cut capacity $\kappa_W$ is a submodular setfunction on the Borel sets in $[0,1]$.

\begin{thm}\label{THM:SK-QK}
Let $(G_n\colon n\in\bbZ_+)$ be a sequence of graphs that is convergent in the dense
sense. Then the sequence $(\kappa_{G_n}\colon n\in\bbZ_+)$ is quotient-convergent.
Furthermore, if $G_n\to W$ for some graphon $W$, then $Q_k(\kappa_{G_n})\hto
Q_k(\kappa_W)$ for every $k\in\bbZ_+$.
\end{thm}

\begin{proof}
Let $\eps>0$. Our goal is to prove that if $G_1$ and $G_2$ are two graphs with
at least $m_0 =32(k+1)/\eps^2$ nodes, then
\begin{equation}\label{EQ:DQ-DG0}
d^\Haus(Q_k(\kappa_{G_1}), Q_k(\kappa_{G_2}))\le \delta_\square(G_1,G_2)+\eps.
\end{equation}
We start with proving similar bounds for special pairs of graphs. First, let
$G_1$ and $G_2$ be two graphs on the same set $V$ of nodes, and let $m=|V|$.
Then
\begin{equation}\label{EQ:DQ-DG}
d^\Haus(Q_k(\kappa_{G_1}), Q_k(\kappa_{G_2}))\le d_\square(G_1,G_2),
\end{equation}
since for every $X\subseteq V$,
\[
|\kappa_{G_1}(X)-\kappa_{G_2}(X)|
= \frac1{m^2}\cdot|e_{G_1}(X,V\setminus X)-e_{G_2}(X,V\setminus X)|\le d_\square(G_1,G_2).
\]
Recall that the graph $G(n)$ is obtained from the graph $G$ by replacing each node $u$
by a set $V_u$ of $n$ non-adjacent copies and connecting all nodes of $V_u$ to all nodes of $V_v$ if $uv \in E(G)$. 
We claim that for every graph $G=(V,E)$ with $m=|V|\geq m_0$ nodes, and every $n\in\bbZ_+$,
\begin{equation}\label{EQ:DQ-DG2}
d^\Haus(Q_k(\kappa_{G}), Q_k(\kappa_{G(n)}))\le\frac{\eps}{4}.
\end{equation}
It is clear that $Q_k(\kappa_G)\subseteq Q_k(\kappa_{G(n)})$. To prove
almost-containment in the other direction, let $X\subseteq V(G(n))$ and let
$x_u=|X\cap V_u|$. Let $Y$ be a random subset obtained from $X$ by replacing
$X\cap V_u$ by $V_u$ with probability $x_u/n$, and by $\emptyset$ with
probability $1-x_u/n$. Then 
\[
\bbE(\kappa_{G(n)}(Y)) = \frac1{n^2\cdot|E(G)|} \cdot\sum_{ij\in E(G)}\big(x_i\cdot(n-x_j)+x_j\cdot(n-x_i)\big) = \kappa_{G(n)}(X).
\]
Azuma's Inequality~\cite{azuma1967weighted} implies that
\[
\Pr\Big(|\kappa_{G(n)}(Y)-\kappa_{G(n)}(X)|> \frac{\eps}{4} \Big) \le 2\cdot e^{-\eps^2\cdot m/32}.
\]
Now if $\PP=(U_1,\dots,U_k)$ is a $k$-partition of $V(G(n))$, then we can
``round'' it to a $k$-partition $\PP'=(U'_1,\dots,U'_k)$ that does not split
any of the sets $V_u$, by placing $V_u$ in class $U'_i$ with probability
$|U_i\cap V_u|/n$. Taking the union bound over all sets that are unions of the
sets $U_i$, we get
\[
\Pr\Big(\|\kappa_{G(n)/\PP}-\kappa_{G(n)/\PP'}\| \ge \frac{\eps}{4}\Big)\le  2^{k+1}\cdot e^{-\eps^2\cdot m/32}<1.
\]
So $G(n)$ has a $k$-partition $\PP'$ with
$\|\kappa_{G(n)/\PP}-\kappa_{G(n)/\PP'}\| \le \eps/4$. The partition $\PP'$
defines a $k$-partition $\PP''$ of $V(G)$ such that
$\kappa_{G(n)/\PP'}=\kappa_{G/\PP''}$. Thus there is a $k$-partition $\PP'$ of
$V(G)$ such that
\[
\|\kappa_{G(n)/\PP}-\kappa_{G/\PP''}\|\le\frac{\eps}{4},
\]
proving \eqref{EQ:DQ-DG2}.

Finally, let $G_i$ be a graph with $n_i\ge m_0$ nodes for $i=1,2$. By the
definition of the $\delta_\square$ distance, there is an integer $p\ge 1$ such
that, for an appropriate identification of the node sets of the graphs
$G_1'=G_1(p\cdot n_2)$ and $G_2'=G_2(p\cdot n_1)$, we have
\[
d_\square(G_1',G_2')\le \delta_\square(G_1,G_2)+\frac{\eps}{4}.
\]
Then, by \eqref{EQ:DQ-DG} and \eqref{EQ:DQ-DG2}, we get
\begin{align*}
&d^\Haus(Q_k(\kappa_{G_1}), Q_k(\kappa_{G_2}))\\
&\quad\le
d^\Haus(Q_k(\kappa_{G_1}), Q_k(\kappa_{G_1'}))+d^\Haus(Q_k(\kappa_{G_2}), Q_k(\kappa_{G_2'}))+d^\Haus(Q_k(\kappa_{G_1'}), Q_k(\kappa_{G_2'}))\\
&\quad\le \frac{\eps}{4}+\frac{\eps}{4}+d_\square(G_1',G_2')\\
&\quad\le \delta_\square(G_1,G_2)+\eps.
\end{align*}
This proves \eqref{EQ:DQ-DG0}. The first assertion of the theorem is then immediate by \eqref{EQ:DQ-DG0}. The proof of the second assertion is similar.
\end{proof}

\begin{rem}\label{REM:QUOTIENTS}
The proof above is quite similar to the proof of~\cite[Theorem 2.8]{BCLSV2}
and of~\cite[Lemma 12.3]{lovasz2012large}. In fact, the quotient sets in the
characterization of \emph{right convergence} are quite similar to the quotient sets
for the cut capacity function. Recall that one can equivalently define the convergence of dense graph sequences using the notion of {\it quotient sets}~\cite{BCLSV2}. A {\it $k$-quotient} $G/\PP$ of a finite graph $G=(V,E)$ is defined by a $k$-partition $\PP=(V_1,\dots,V_k)$ of $V$. It is a weighted graph on $[k]$ with node weights
\[
\alpha(i)=\alpha_{G/\PP}(i)=\frac{|V_i|}{|V|} \qquad \text{for all $i\in[k]$}
\]
and edge weights
\[
\beta(ij)=\beta_{G/\PP}(ij)=\frac{e_G(V_i,V_j)}{|V_i|\cdot|V_j|} \qquad \text{for all $i,j\in[k]$}.
\]
We represent $G/\PP$ as a vector in $\mathbb{R}^m$, where $m = 2k + \binom{k}{2}$. The coordinates correspond to the $\binom{k}{2}$ edge weights, $k$ loop weights, and $k$ node weights. Let $S_k(G)$ denote the set of all $k$-quotients of $G$. This set can be viewed as a subset of the space $\mathbb{R}^m$. A sequence of graphs $(G_n \colon n \in \mathbb{Z}_+)$ is said to be {\it right-convergent} if $|V(G_n)| \to \infty$ and the sequence of sets $S_k(G_n)$ converges in Hausdorff distance for all $k\in\bbZ_+$.

In the light of this similarity, it seems natural to ask whether a converse to Theorem~\ref{THM:SK-QK} holds. That is, assume that $(G_n\colon n \in \mathbb{Z}_+)$ is a sequence of graphs with $|V(G_n)| \to \infty$, and that $(\kappa_{G_n} \colon n \in \mathbb{Z}_+)$ is quotient-convergent. Does it then follow that $(G_n \colon n \in \mathbb{Z}_+)$ is convergent in the dense sense? To motivate this question, we consider a modified version of quotients. Let us introduce the weights
\[
\gamma(ij)=\gamma_{G/\PP}(ij)=\alpha(i)\cdot \alpha(j)\cdot \beta(ij)=\frac{e_G(V_i,V_j)}{|V|^2}.
\]
We change the edge weights of $G/\PP$ to $\gamma(ij)$, and ignore the node weights;
let us denote the edge-weighted graph obtained this way by $G \PER\PP$. We denote by
$S'_k(G)$ the set of modified $k$-quotients $G \PER\PP$. The set $S_k'(G)$ can
be considered as a subset of $\R^M$, where $M=\binom{k+1}{2}$.

It is an easy observation that for every graph $G$ and every $k\in\bbZ_+$, the
quotient sets $S'_k(G)$ and $Q_k(\kappa_G)$ carry the same information. More
exactly, for every partition $\PP=(V_1,\dots,V_k)$ of $V$, the edge-weighted graph $G\PER\PP$ and the capacity function
$\kappa_G/\PP$ can be easily computed from each other. Indeed, for every $A\subseteq [k]$,
\begin{equation}\label{EQ:EDGE2CUT}
\kappa_G/\PP(A)=\sum_{i\in A}\sum_{j\in [k]\setminus A} \gamma(ij),
\end{equation}
and for every $i,j\in[k]$,
\begin{equation}\label{EQ:EDGE2CUT2}
\gamma(ij)=\frac12\cdot (\kappa_G(V_i)+\kappa_G(V_j)-\kappa_G(V_i\cup V_j).
\end{equation}
It is also clear that $S'_k(G)$ can be computed from $S_k(G)$, and so
$Q_k(\kappa_G)$ can be computed from $S_k(G)$. With some care, this observation
can be made into another proof of Theorem~\ref{THM:SK-QK}.

Clearly, the node weights
$\alpha_i$ of $G/\PP$ cannot be computed from the edge weights $\gamma(ij)$ of
$G\PER\PP$. However, it remains open whether the whole set $S_k(G)$ is determined by $S_k'(G)$, at
least approximately, for large graphs $G$. This would still give a converse to
Theorem~\ref{THM:SK-QK}.
\end{rem}

\subsection{Homomorphism densities}\label{SSEC:HOMO-DENSITY}
\label{sec:hom}

Let $G=(V,E)$ be a graph. For $X\subseteq E(G)$, let $G_X=(V(G),X)$. It is easy to see that for fixed graphs $F$ and $G$, the homomorphism density $t(F,G_X)$, as a function of $X\subseteq E(G)$, is supermodular and increasing, so the setfunction $\tau_{F,G}\colon 2^{E(G)}\to[0,1]$ defined as $\tau_{F,G}(X)=1-t(F,G_{X^c})$ is submodular and increasing. We can extend this to a graphon $W\colon [0,1]^2\to[0,1]$ instead of $G$ as follows: for every Borel subset $X\subseteq [0,1]^2$, we denote by $W_X$ the graphon obtained by reducing $W$ to $0$ outside $X$ and set $\tau_{F,W}(X)=1-t(F,W_X)$. 

The following theorem is proved in~\cite{BBLTconv}, using results from the theory of graph property testing.

\begin{thm}\label{THM:DENSE-CONV-QUOTIENT}
If $(G_n\colon  n\in\bbZ_+)$ is a sequence of graphs that converges to a graphon $W$ in the sense of dense graph convergence, then for every graph $F$, the sequence of setfunctions $\tau_{F,G_n}$ quotient-converges to $\tau_{F,W}$. 
\end{thm}    

\paragraph{Acknowledgements.} The research was supported by the Ministry of Innovation and Technology of Hungary from the National Research, Development and Innovation Fund -- grant numbers EKÖP-24 (Márton Borbényi), KKP-139502 and STARTING 150723 (László Márton Tóth), ADVANCED 150556 and ELTE TKP 2021-NKTA-62 (Kristóf Bérczi), by the Lend\"ulet Programme of the Hungarian Academy of Sciences -- grant number LP2021-1/2021, and by Dynasnet European Research Council Synergy project -- grant number ERC-2018-SYG 810115.

\bibliographystyle{abbrv}
\bibliography{sequences}

\end{document}